\documentclass[article,12pt]{amsart}
\usepackage{lineno,hyperref,enumitem}
\usepackage{amsmath,amssymb,dsfont}

\newtheorem{thm}{Theorem}[section]
\newtheorem{defn}[thm]{Definition}
\newtheorem{prop}[thm]{Proposition}

\begin{document}

\title[Characterization of multiframelet set]{Characterization of multiframelet set on local fields of positive characteristic}

\author[D Haldar]{Debasis Haldar}

\address{Department of Mathematics\\ NIT Rourkela\\ Rourkela 769008\\ India}

\email{debamath.haldar@gmail.com}





\begin{abstract}
This paper presents a discussion on multiframelet set, multiwavelet set and set correspond to super wavelet on local fields of positive characteristic. We characterize Parseval multiframelet set and give equivalent conditions multiwavelet set holds. Furthurmore, we characterize MRA multiwavelet set with the help of dimension function.  
\end{abstract}

\maketitle

\section{Introduction}

Motivated from the recent works of Behara and Jahan \cite{BJ15, BJ12}, Shukla and Maury \cite{SM17}, our aim is to develop the theory of multiwavelet set, scaling set, Parseval scaling set, Parseval framelet set in local field of positive characteristic(LFPC).

J Benedetto and R Benedetto has discussed the existance of wavelet sets for locally compact abelian group (LCAG) and related groups in \cite{BB04,B04} and Curry and Mayeli \cite{CM11} has discussed same. We characterize finite ordered multiwavelet set in LFPC. We also give some equivalent conditions of finite ordered multiwavelet set in LFPC. Moreover, we give characterization of Parseval frame set and equivalent condition of set induced from a Besselet family with bound 1.

The dimension function of multiwavelets was introduced and studied by Auschar \cite{A95}. Later Bownik et al. \cite{BRS01} have studied this for MRA and non-MRA wavelets over $L^{2}(\mathbb{R})$.

The remainder of the paper is structured as follows. In `Preliminaries' section, we study algebric and topological property of the local fields and multiframelets needed throughout the paper. We characterize Parseval frame set and give some equivalent condition of multiwavelet set in section `Multiframelet set and multiwavelet sets'. In section `Multiresolution analysis', we discuss about scaling set and MRA multiwavelet set. Finally in `Super-wavelet' section, we discuss about decomposable Parseval framelet and equivalency of two Parseval frame super-wavelet.

\section{Preliminaries}
We will discuss in this section some preliminary results required for rest of this article.

\subsection{Local fields}
Throughout this article, $\mathfrak{L}$ denotes the set $\{1, \ldots, L\}$, $\mathds{1}_S$ denotes characteristic function of the set $S$, $K$ denotes a local field and a topological space. By a locally compact field or, in short, a local field we mean a locally compact, complete, non-discrete and totally disconnected field. $K^*$ and $K^+$ denote as the multiplicative and additive groups of $K$ respectively.\\

Readers are referred to \cite{RV99,T75} for proofs of all the results stated in this section. The valuation $|x|$ of $x \in K$ satisfies the following properties
\begin{enumerate}[label=(\roman*)]
	\item $|x| \geq 0$ and $|x|=0$ if and only if $x=0$;
	\item $|xy|=|x||y|$;
	\item $|x+y| \leq$ max $\{|x|,|y|\}$. $|x+y|=\{|x|,|y|\}$ if $|x| \neq |y|$.
\end{enumerate}
Property (iii) is called the ultrametric inequality and hence $K$ is an non-Archimedean field. Define $\mathfrak{B} =\{x\in K : |x|<1\}$ and $\mathfrak{D} =\{x\in K : |x|\leq 1\}$. Then they both are compact and open. $\mathfrak{B}$ is the prime ideal in $K$ and the ring of integers $\mathfrak{D}$ is the unique maximal compact subring of $K$. \\

The set of valuation of the elements of $K$ is a discrete set of the form $\{v^s : s \in \mathbb{Z}\} \cup \{0\}$ for some positive $v$ as $K$ is totally disconnected. Let $\mathfrak{p}$ be an element of maximum valuation in $\mathfrak{B}$. Then $\mathfrak{B}=\langle \mathfrak{p}\rangle =\mathfrak{p}\mathfrak{D}$. It is not difficult to see that $\mathfrak{D}$ is compact and open, hence $\mathfrak{B}$ is so. Therefore $\mathfrak{D}/\mathfrak{B} \cong GF(q)$, where $q=p^c$ for some prime $p$ and $c \in \mathbb{N}$ as $GF(q)$ is a vector space of $c$-dimension over $GF(p)$. Let $\mathfrak{B}^s =\mathfrak{p}^s \mathfrak{D} = \{x \in K : |x| \leq q^{-s}\}$ for $s \in \mathbb{Z}$. These $\mathfrak{B}^s$ are said to be fractional ideals which are also compact open subgroup of $K^+$.\\

 $K^+$ is being a LCAG, one may choose a Haar measure $dx$ for $K^+$. $d(\alpha x)$ is also a Haar measure and $d(\alpha x)=|\alpha|dx$ for $\alpha \in K \setminus \{0\}$. This Haar measure is normalized by $\int \limits_{K} \mathds{1}_{\mathfrak{D}}(x) dx =1$. If $dx$ is a Haar measure on $K^+$, then Haar measure on $K^*$ is given by $\frac{dx}{|x|}$. It is easy to see the following
\begin{itemize}
	\item $|\mathfrak{p}|=q^{-1} =\mu(\mathfrak{B})$.
	\item valuation set of $K$ is $\{q^s : s \in \mathbb{Z}\} \cup \{0\}$.
\end{itemize}
The Hilbert space of all complex-valued functions on $K$, square integrable with respect to the
measure $dx$, is denoted by $L^{2}(K)$. The inner product in this space is given by $$\langle f,g \rangle = \int \limits_{K} f(x)\overline{g(x)} dx~~~~~ \text{  for } f,g \in L^{2}(K).$$

Any $x \in \mathfrak{B}^s$ can be uniquely expressed as $$x=\sum \limits_{n=s}^\infty a_n \mathfrak{p}^n ,~~~~~~~~~ a_n \in \mathfrak{U},~ a_s \neq 0$$ where $\mathfrak{U}=\{c_0, c_1, \ldots, c_{q-1}\}$ is a fixed full set of coset representatives of $\mathfrak{B}$ in $\mathfrak{D}$. In this case, $|x|=q^{-s}$. As $K$ is a local field, consider $\chi$ be a fixed non-trivial, unitary, continuous character on $K^+$ which is trivial on $\mathfrak{D}$ but non-trivial on $\mathfrak{B}^{-1}$. Fix $y \in K$, we define $\chi_y (x)=\chi(yx)$ for $x\in K$. This character satisfy following important result given by Taibleson \cite{T75}.
\begin{prop}
	If $\{u(n)\}_{n=0}^\infty$ is complete list of distinct coset representatives of $\mathfrak{D}$ in $K^+$, then $\{\chi_{u(n)}\}_{n=0}^\infty$ is a complete list of distinct characters on $\mathfrak{D}$. Further, it is a complete orthonormal system on $\mathfrak{D}$.
\end{prop}

 The Fourier series of $f \in L^{1}(\mathfrak{D})$ is defined as $\sum \limits_{n=0}^{\infty} \hat{f}(u(n))\chi_{u(n)}$ where Fourier coefficients are given by $\hat{f}(u(n)):=\int \limits_{\mathfrak{D}}f(x)\overline{\chi_{u(n)}(x)} dx$.
 
\begin{defn}(Fourier transform).
	For $f \in L^{1}(K) \cap L^{2}(K)$, the Fourier transform $\mathcal{F}$ is defined by $$\mathcal{F}[f](\xi):= \hat{f}(\xi) = \int \limits_{K} f(x)\overline{\chi_\xi (x)} dx = \int \limits_{K} f(x)\chi(- \xi x) dx.$$
\end{defn}
This Fourier transform satisfy following
\begin{itemize}
	\item $\mathcal{F} : L^{1}(K) \rightarrow L^{\infty}(K)$ is a bounded linear transformation satisfying $\|\hat{f}\|_{\infty} \leq \|f\|_1$.
	\item If $f \in L^{1}(K)$ then $\hat{f}$ is uniformly continuous on $K$.
	\item \textbf{Plancheral Theorem :} $\|\hat{f}\|_{2} = \|f\|_{2}$, for $f \in L^{2}(K)$.
	\item \textbf{Parseval Theorem :} $\langle f,g \rangle = \langle \hat{f}, \hat{g} \rangle$, for all $f,g \in L^{2}(K)$.
	\item \textbf{Riemann-Lebesgue Theorem :} If $f \in L^{1}(K)$, then $\hat{f}(\xi) \rightarrow 0$ as $|\xi| \rightarrow \infty$.
\end{itemize}

In order to impose `natural' order on $\{u(n)\}_{n=0}^\infty$ used to develop the Fourier series theory on $L^{2}(\mathfrak{D})$, we use a set $\{1=\epsilon_0, \epsilon_{1}, \ldots, \epsilon_{c-1}\} \subset \mathfrak{D}^{*}=\mathfrak{D}/\mathfrak{B}$ such that span$\{\epsilon_0, \ldots, \epsilon_{c-1}\} \cong GF(q)$. Let $\mathbb{N}_{0}=\mathbb{N} \cup \{0\}$. We now give an example of a character $\chi$ on $K$ that is trivial on $\mathfrak{D}$ but non-trivial on $\mathfrak{B}^{-1}$ as follows \cite{Z97}
\begin{center}
$\chi(\epsilon_{k}\mathfrak{p}^{-j})= \begin{cases}
e^{2 \pi\frac{i}{p}} & \mbox{if $k=0,~ j=1$};\\
1 & \mbox{if $k=1, \ldots, c-1$ or $j \neq 1$.}
\end{cases}$ 
\end{center}

\begin{defn}
	Any $n\in \mathbb{N}_{0} \cap [0,q)$ can be written as $n=\sum \limits_{k=0}^{c-1} a_{k} p^{k},~~~~~~~\text{where } 0\leq a_{k} <p.$
	Define $u(n)=(\sum \limits_{k=0}^{c-1} a_{k} \epsilon_{k})\mathfrak{p}^{-1}$.
	Note that $\{u(n)\}_{n=0}^{q-1}$ is a complete set of coset representatives of $\mathfrak{D}$ in $\mathfrak{B}^{-1}$. Now for $n \in \mathbb{N}_{0} \cap [q,\infty)$, we have, $n=\sum \limits_{k=0}^{s} b_{k} q^{k},~~~~~~~ 0\leq b_{k} <q.$
	Define $u(n)=\sum \limits_{k=0}^{s} \mathfrak{p}^{-k} u(b_k)$.
\end{defn} 
This defines $u(n)$ for all $n \in \mathbb{N}_0$. In general that $u$ is not linear but satisfies $$u(rq^{k}+s)=\mathfrak{p}^{-k}u(r)+u(s)~~ \text{where } r,k \in \mathbb{N}_{0}, 0\leq s< q^{k}.$$
In the following proposition, Behera and Jahan \cite{BJ12} proved important properties of aforementioned $u(n)$ which has repeated application rest of the paper.
\begin{prop} \label{subgroup}
	Let $u(n)$ be defined as above for all $n\in \mathbb{N}_{0}$.
	\begin{enumerate}[label=(\alph*)]
		\item $u(n)=0$ if and only if $n=0$. If $k\geq 1$, then we have $|u(n)|=q^{k}$ if and only if $q^{k-1} \leq n<q^k$.
		\item $\{u(n)\}_{n=0}^\infty =\{-u(n)\}_{n=0}^\infty$.
		\item For a fixed $m \in \mathbb{N}_{0}$, we have $\{u(m)+u(n)\}_{n=0}^\infty =\{u(n)\}_{n=0}^\infty$.
	\end{enumerate}
\end{prop}
It clearly follows from Proposition \ref{subgroup} that $\{u(n)\}_{n=0}^\infty$ is a subgroup of $K^+$. Since $\bigcup \limits_{j\in \mathbb{Z}} \mathfrak{p}^{-j}\mathfrak{D}=K$, we shall use $\mathfrak{p}^{-1}$ as the dilation and since $\{u(n)\}_{n=0}^\infty$ is a complete list of distinct coset representatives of $\mathfrak{D}$ in $K$, we shall use $\{u(n)\}_{n=0}^\infty$ as the translation set in rest of this article.\\
A function $f \in L^{2}(K)$ is said to be integral periodic if $$f(\cdot + u(n))=f(\cdot), \hspace{0.5cm} \forall n \in \mathbb{N}_{0}.$$

\subsection{Multiframelets}

\begin{defn}(Multiwavelet).
	A set of functions $\Psi = \{\psi^1, \ldots, \psi^L\} \subset L^{2}(K)$ is said to be a multiwavelet of order $L$ if affine system $$\mathcal{W}(\Psi):=\{\psi^{l}_{j,t}:=q^{\frac{j}{2}}\psi^{l}(\mathfrak{p}^{-j} \cdot -u(t)) : j \in \mathbb{Z},~ t \in \mathbb{N}_{0},~ l \in \mathfrak{L}\}$$ is an orthonormal basis for $L^{2}(K)$.
\end{defn}

Note that for $f \in L^{2}(K)$ and $\xi \in K$, $$\widehat{f_{j,t}}{(\xi)} = q^{-\frac{j}{2}} \chi_{u(t)}(-\mathfrak{p}^{j}\xi)\hat{f}(\mathfrak{p}^{j}\xi).$$

\begin{defn}(Multiwavelet Set).
	A measurable set $W \subset K$ is said to be a multiwavelet set of order $L$ if $W = \bigcup \limits_{l=1}^{L} W_l$ for measurable sets $W_l$ and $\hat{\psi}^{l}=\mathds{1}_{W_{l}}$ where $\{\psi^1, \ldots, \psi^L\}$ is a multiwavelet for $L^{2}(K)$.
\end{defn}
When $L=1$, $W$ is simply said to be a wavelet set.
 
\begin{defn}(Multiframelet).
	For a set of basic functions $\mathcal{F}=\{f^{(1)}, \ldots, f^{(L)}\} \subset L^{2}(K)$, affine system $\mathcal{W}(\mathcal{F})$ is said to be a multiframelet of order $L$ if there exist finite $A,B > 0$ so that for all $g \in L^{2}(K)$ we have,
	\begin{equation} \label{dmf}
	A\left \| g \right \|^{2} \leq \sum \limits_{ l \in \mathfrak{L}} \sum \limits_{j \in \mathbb{Z}} \sum \limits_{t \in \mathbb{N}_{0}}| \langle g,f^{(l)}_{j,t} \rangle |^{2} \leq B\left \| g \right \|^{2}.
	\end{equation}
\end{defn}

When $L=1$, it is simply said to be a framelet. $\mathcal{W}(\mathcal{F})$ is said to be a besselet with bound $B$ if it satisfy later inequality of equation \eqref{dmf}. $\mathcal{W}(\mathcal{F})$ is called a tight multiframelet if $A=B$ and Parseval multiframelet if $A=1=B$ in equation \eqref{dmf}. It is well-known that $\mathcal{W}(\mathcal{F})$ is a Parseval multiframelet for $L^{2}(K)$ if and only if $g=\sum \limits_{ l \in \mathfrak{L}} \sum \limits_{j \in \mathbb{Z}} \sum \limits_{t \in \mathbb{N}_{0}} \langle g,f^{(l)}_{j,t} \rangle f^{(l)}_{j,t}$ for all $g \in L^{2}(K)$. Therefore, Parseval multiframelet is a generalization of the concept of orthonormal basis to a system having no minimality condition. Similar to multiwavelet set, we define multiframelet set.

\begin{defn}(Multiframelet set).
	A measurable set $F \subset K$ is said to be a multiframelet set/ Parseval multiframelet set of order $L$ if $F=\bigcup \limits_{l=1}^{L} F_l$ for measurable sets $F_l$ where $\hat{f^{(l)}}=\mathds{1}_{F_l}$ and $\{f^{(l)} : l=1, \ldots,L\}$ generates a multiframelet/ Parseval multiframelet for $L^{2}(K)$.
\end{defn}

The following theorem is a characterization of $\mathcal{W}(\Psi)$ to be a Parseval multiframelet for $L^{2}(K)$, given by Behera and Jahan \cite{BJ15}.

\begin{thm} \label{char}
	Suppose $\Psi = \{\psi^{(m)} : m = 1, 2, \ldots, L\} \subset L^{2}(K)$. Then the affine system $\mathcal{W}(\Psi)$ is a Parseval multiframelet for $L^{2}(K)$ if and only if the following equalities holds for $\xi \in K$ a.e.
	
	\begin{enumerate}[label=(\roman*)]
		\item \begin{equation} \label{ort}
			\sum \limits_{m=1}^L \sum \limits_{j \in \mathbb{Z}} \lvert \widehat{\psi^{(m)}}(\mathfrak{p}^{-j}\xi) \rvert^{2} =1.
		\end{equation}
		\item $\sum \limits_{m=1}^L \sum \limits_{j \in \mathbb{N}_{0}} \widehat{\psi^{(m)}}(\mathfrak{p}^{-j}\xi) \overline{ \widehat{\psi^{(m)}}(\mathfrak{p}^{-j}(\xi+u(t)))}=0 ~~~~~~~~~~\text{     for } t \in \mathbb{N}_{0} \setminus q\mathbb{N}_{0}$.
	\end{enumerate}
	In particular, $\Psi$ is an multiwavelet of order $L$ in $L^{2}(K)$ if and only if $\| \psi^{(m)} \|=1$ for all $1 \leq m \leq L$ and satisfy above two equations.	
\end{thm}

Behera and Jahan have given following three equivalent conditions of Parseval multiframelet in \cite{BJ15}.

\begin{thm} \label{eqiv}
	Suppose $\Psi = \{\psi^{(m)} : m = 1, 2, \ldots, L\} \subset L^{2}(K)$. Assume that $\mathcal{W}(\Psi)$ is a besselet with bound 1. Then the following are equivalent
	\begin{enumerate}
		\item $\mathcal{W}(\Psi)$ is a Parseval multiframelet.
		\item $\Psi$ satisfies $\sum \limits_{m=1}^L \sum \limits_{j \in \mathbb{Z}} \lvert \widehat{\psi
			^{(m)}}(\mathfrak{p}^{-j}\xi) \rvert^{2} =1$ for $\xi \in K$ a.e.
		\item $\Psi$ satisfies 
		\begin{equation} \label{setint}
			\sum \limits_{m=1}^{L} \int \limits_{K} \frac{\lvert \widehat{\psi^{(m)}} (\xi) \rvert^2}{\lvert \xi \rvert} d\xi = \frac{q-1}{q}. 
		\end{equation}
	\end{enumerate}
\end{thm}

\begin{thm}\cite{BJ15} \label{cd}
	Suppose $\Psi =\{\psi^1, \ldots, \psi^L\} \subset L^{2}(K)$. Then the following are equivalent
	\begin{enumerate} [label=(\roman*)]
		\item $\Psi$ is a multiwavelet of $L^{2}(K)$.
		\item $\Psi$ satisfy equations \eqref{charofortho} and \eqref{ort}.
		\item $\Psi$ satisfy equations \eqref{charofortho} and \eqref{setint}.
	\end{enumerate}
\end{thm}

\section{Multiframelet set and Multiwavelet set}

\begin{thm}
Let $\Psi = \{\psi_1, \ldots, \psi_L\} \subset L^{2}(K)$ such that $\widehat{\psi_m} = \mathds{1}_{W_m}$, where $\bigcup \limits_{m=1}^{L} W_m \subset K$ is measurable. If the affine system $\mathcal{W}(\Psi)$ is an orthonormal in $L^{2}(K)$ then
\begin{enumerate}[label=(\roman*)]
 \item for $m=1, \ldots, L;~~ \{ W_m + u(t) : t \in \mathbb{N}_{0}\}$ is a partition of $K$.
 \item $\{W_m : m=1,\ldots, L\}$ is a partition of a subset of $K$.
 \item for $j \geq 1$ and $l,m \in \{1,\ldots, L\},~~ W_l \cap \mathfrak{p}^{j}W_m = \emptyset.$
\end{enumerate}
\end{thm}

\begin{proof}
	As $\mathcal{W}(\Psi)$ is an orthonormal in $L^{2}(K)$
	\begin{equation} \label{charofortho} 
		\sum \limits_{t \in \mathbb{N}_{0}} \widehat{\psi_l}(\xi + u(t)) \overline{\widehat{\psi_m}(\mathfrak{p}^{-j}(\xi+u(t)))}= \delta_{j,0}~ \delta_{l,m}
	\end{equation}
	for $\xi \in K$ a.e., $1 \leq l,m \leq L$, $j \geq 0$.
Putting $\widehat{\psi_n} = \mathds{1}_{W_n}$ in equation \eqref{charofortho}, we get
\begin{equation*} \label{eqivset}
	\sum \limits_{t \in \mathbb{N}_{0}} \mathds{1}_{W_l \cap \mathfrak{p}^j W_m}(\xi +u(t)) = \delta_{j,0}~ \delta_{l,m}
\end{equation*}
for $\xi \in K$ a.e., $1 \leq l,m \leq L$ and $j \geq 0$.
\begin{itemize}
	\item When $j=0$ : $\sum \limits_{t \in \mathbb{N}_{0}} \mathds{1}_{W_l \cap  W_m}(\xi+u(t)) = \delta_{l,m}$.
	\begin{itemize}
		\item if $l=m$, $~\sum \limits_{t \in \mathbb{N}_{0}} \mathds{1}_{W_m}(\xi+u(t)) =1$.\\ Hence $\{ W_m + u(t) : t \in \mathbb{N}_{0}\}$ is a partition of $K$ for $m \in \{1, \ldots, L\}$.
		\item if $l \neq m$, $~ \sum \limits_{t \in \mathbb{N}_{0}} \mathds{1}_{W_l \cap W_m}(\xi+u(t)) = 0$.\\ Hence $\{ W_m : m= 1, \ldots, L\}$ is a partition of a subset of $K$.
	\end{itemize}

   \item When $j \neq 0$ : 
   \begin{eqnarray*} 
   	 && \sum \limits_{t \in \mathbb{N}_{0}} \mathds{1}_{W_l \cap \mathfrak{p}^j W_m}(\xi +u(t)) = 0 \\
   	&\Rightarrow& \mathds{1}_{W_l \cap \mathfrak{p}^j W_m}(\xi +u(t)) = 0, ~~ \forall~ t \in \mathbb{N}_0 \\
   	&\Rightarrow& W_l \cap \mathfrak{p}^{j}W_m=\emptyset
   \end{eqnarray*}
\end{itemize}
This completes the proof.
\end{proof}

We now characterize Parseval multiframelet set in $K$ with the help of Theorem \ref{char}. 

\begin{thm} \label{ps}
	A set $\mathcal{P} = \bigcup \limits_{m=1}^{L} W_m \subset K$ be a Parseval multiframelet set of order $L$ for $L^{2}(K)$ if and only if 
	\begin{enumerate}[label=(\alph*)]
		\item $\{\mathfrak{p}^{j}W_{m} : j \in \mathbb{Z},~ m = 1, \ldots, L\}$ is a measurable partition of $K$ a.e.
		\item for all $m = 1, \ldots, L$ and for all $j \in \mathbb{N}_{0},~~~  \mathfrak{p}^{j}W_{m} \cap [\mathfrak{p}^{j}W_{m}+u(t)] = \emptyset$, where $t \in \mathbb{N}_{0} \setminus q\mathbb{N}_{0}$.
	\end{enumerate}
\end{thm}

\begin{proof}
As $\mathcal{P}$ is a Parseval multiframelet set, there exist $\Psi = \{\psi^{(m)} : m = 1, 2, \ldots, L\} \subset L^{2}(K)$ such that $\mathcal{W}(\Psi)$ is a Parseval multiframelet for $L^{2}(K)$ where $\widehat{\psi^{(m)}} = \mathds{1}_{W_m}~,~ \forall ~m$.\\
From $(i)$ of Theorem \ref{char}, we have, for $\xi \in K$ a.e.
\begin{eqnarray*}
	\sum \limits_{m=1}^L \sum \limits_{j \in \mathbb{Z}} \lvert \mathds{1}_{W_m}(\mathfrak{p}^{-j}\xi) \rvert^{2} =1 \iff \sum \limits_{m=1}^L \sum \limits_{j \in \mathbb{Z}}  \mathds{1}_{\mathfrak{p}^{j}W_m}(\xi) =1.
\end{eqnarray*}

 Hence $\{\mathfrak{p}^{j}W_{m} : j \in \mathbb{Z},~ m = 1, \ldots, L\}$ is a measurable partition of $K$ a.e.\\
Also from $(ii)$ of Theorem \ref{char}, we get, for $\xi \in K$ a.e. and $t \in \mathbb{N}_{0} \setminus q\mathbb{N}_{0}$,

\begin{eqnarray*}
 \sum \limits_{m=1}^L \sum \limits_{j \in \mathbb{N}_0} \mathds{1}_{\mathfrak{p}^{j}W_m}(\xi) \mathds{1}_{\mathfrak{p}^{j}W_m}(\xi + u(t)) =0
\iff \sum \limits_{m=1}^L \sum \limits_{j \in \mathbb{N}_0} \mathds{1}_{\mathfrak{p}^{j}W_m \cap [\mathfrak{p}^{j}W_m + u(t)]}(\xi) =0
\end{eqnarray*}

Therefore for $m \in \{1, \ldots, L\},~j \in \mathbb{N}_0$ and $t \in \mathbb{N}_{0} \setminus q\mathbb{N}_{0}$, 
$$\mathfrak{p}^{j}W_m \cap [\mathfrak{p}^{j}W_m + u(t)]= \emptyset.$$

Converse is easy to see. In particular, $\mathcal{P}$ is a multiwavelet set implies $\mu(W_m)=1$ for all $m$, in addition satisfying above two conditions.
\end{proof}

We now give three equivalent conditions for Parseval multiframelet set.

\begin{thm} \label{bes}
	Let $\Psi = \{\psi^{(1)}, \ldots \psi^{(L)}\} \subset L^{2}(K)$ with $\widehat{\psi^{(m)}} = \mathds{1}_{W_m}$, where $\bigcup \limits_{m=1}^{L} W_m \subset K$ is measurable. If $\mathcal{W}(\Psi)$ is a besselet with bound 1, then the following statements are equivalent
	\begin{enumerate}
		\item $\|f\|^2 = \sum \limits_{m=1}^L \sum \limits_{j \in \mathbb{Z}} \sum \limits_{t \in \mathbb{N}_{0}} q^{-j} \left \lvert ~ \int \limits_{\mathfrak{p}^{-j}W_m} \hat{f}(\xi) \chi_{u(t)}(\mathfrak{p}^{-j}\xi) d\xi \right \rvert^{2},~~ \forall ~ f \in L^{2}(K)$.
		\item $\{\mathfrak{p}^{j}W_{m} : j \in \mathbb{Z},~ m = 1, \ldots, L\}$ is a measurable partition of $K$ a.e.
		\item $\int \limits_{\bigcup \limits_{m=1}^{L}W_m} \frac{1}{|\xi|} d\xi = \frac{q-1}{q}$.
	\end{enumerate}
\end{thm}

\begin{proof}
	Using (1) of Theorem \ref{eqiv} and applying Parseval theorem, we obtain, for all $f \in L^{2}(K)$,
	\begin{eqnarray*}
		\|f\|^2 &=& \sum \limits_{m=1}^L \sum \limits_{j \in \mathbb{Z}} \sum \limits_{t \in \mathbb{N}_{0}} | \langle f, \psi^{(m)}_{j,t} \rangle|^2 \\
		&=& \sum \limits_{m=1}^L \sum \limits_{j \in \mathbb{Z}} \sum \limits_{t \in \mathbb{N}_{0}} | \langle \hat{f}, \widehat{\psi^{(m)}_{j,t}} \rangle|^2 \\
		&=& \sum \limits_{m=1}^L \sum \limits_{j \in \mathbb{Z}} \sum \limits_{t \in \mathbb{N}_{0}}  \left \lvert ~ q^{-\frac{j}{2}} \int \limits_{K} \hat{f}(\xi) \overline{ \chi_{u(t)}(-\mathfrak{p}^{j}\xi) \mathds{1}_{W_m}(\mathfrak{p}^{j}\xi)} d\xi \right \rvert^{2}\\
		&=& \sum \limits_{m=1}^L \sum \limits_{j \in \mathbb{Z}} \sum \limits_{t \in \mathbb{N}_{0}} q^{-j} \left \lvert ~ \int \limits_{\mathfrak{p}^{-j}W_m} \hat{f}(\xi) \chi_{u(t)}(\mathfrak{p}^{j}\xi) d\xi \right \rvert^{2}.
	\end{eqnarray*}

\noindent (2) follows from Theorem \ref{eqiv}.
Now using (3) of Theorem \ref{eqiv}, we get
$$\sum \limits_{m=1}^L \int \limits_{K} \frac{\mathds{1}_{W_{m}}(\xi)}{|\xi|} d\xi = \int \limits_{\bigcup \limits_{m=1}^{L}W_m} \frac{1}{|\xi|} d\xi = \frac{q-1}{q}.$$
So our assertion is tenable.
\end{proof}

\begin{thm}
	Let $\Psi =\{\psi^1, \ldots, \psi^L\} \subset L^{2}(K)$ such that $\widehat{\psi^m} = \mathds{1}_{W_m}$, where $W=\bigcup \limits_{m=1}^{L} W_m$ is a measurable subset of $K$. Then the following are equivalent
	\begin{enumerate}
	\item $W$ is a multiwavelet set in $K$.
	\item $W$ satisfies 
	\begin{enumerate}
		\item equation \eqref{eqivset},
		\item $\{\mathfrak{p}^{j}W_{m} : j \in \mathbb{Z},~ m = 1, \ldots, L\}$ is a measurable partition of $K$ a.e.
	 \end{enumerate}
 
     \item $W$ satisfies \begin{enumerate}
     	\item equation \eqref{eqivset},
     	\item $\int \limits_{W} \frac{1}{|\xi|} d\xi = \frac{q-1}{q}$.
     \end{enumerate}
	\end{enumerate}
\end{thm}

\begin{proof}
This clearly follows from Theorem \ref{bes}, Theorem \ref{cd} with the help of equation \eqref{eqivset}.
\end{proof}

\section{Multiresolution Analysis}
Wavelet on local field can be constructed from a multiresolution analysis (MRA) similar to $\mathbb{R}$. Jiang et al. \cite{JLJ04} define an MRA and construct orthogonal wavelet on local field of positive characteristic. We now discuss MRA and its properties.

\begin{defn}(Multiresolution analysis).
Let $K$ be a local field of positive characteristic. A multiresolution analysis of $L^{2}(K)$ is an sequence $\{V_j : j \in \mathbb{Z}\}$ of closed subspaces of $L^{2}(K)$ satisfying following conditions
\begin{enumerate}[label=(\roman*)]
	\item $V_{j} \subset V_{j+1}$ for all $j \in \mathbb{Z}$,
	\item $\overline{\bigcup \limits_{j \in \mathbb{Z}} V_j} = L^{2}(K)$,
	\item $\bigcap \limits_{j \in \mathbb{Z}} V_j = \{0\}$,
	\item $f \in V_j$ if and only if $f(\mathfrak{p}^{-1} \cdot) \in V_{j+1}$ for all $j \in \mathbb{Z}$,
	\item there exist a function $\phi \in V_0$, namely the scaling function, such that $\{\phi (\cdot - u(t)) : t \in \mathbb{N}_0\}$ is an orthonormal basis for $V_0$.\\ \\
	
	\noindent If we replace (v) by (vi) as follows, then the sequence $\{V_j : j \in \mathbb{Z}\}$ is said to be a Parseval multiresolution analysis (PMRA).
	\item there is a function $\phi \in V_0$, namely the Parseval scaling function, such that $\{\phi (\cdot - u(t)) : t \in \mathbb{N}_0\}$ is an Parseval frame for $V_0$.
\end{enumerate}
\end{defn}

A measurable set $S \subset K$ is a scaling set/Parseval scaling set if $\hat{\phi} = \mathds{1}_S$ for some scaling function/Parseval scaling function $\phi$.

\begin{thm}\cite{BJ12} \label{sc}
 A function $\phi \in L^{2}(K)$ is a scaling function for a multiresolution analysis of $L^{2}(K)$ if and only if
 \begin{enumerate}
 	\item $\sum \limits_{t \in \mathbb{N}_{0}} \lvert \hat{\phi}(\xi + u(t)) \rvert^{2} = 1$ for $\xi \in \mathfrak{D}$ a.e.
 	\item $\lim\limits_{j \to \infty} \lvert \hat{\phi}(\mathfrak{p}^{j}\xi) \rvert =1$ for $\xi \in K$ a.e.
 	\item there exist an integral periodic function $m_{0} \in L^{2}(\mathfrak{D})$ such that $$\hat{\phi}(\xi) = m_{0}(\mathfrak{p}\xi) \hat{\phi}(\mathfrak{p}\xi) \hspace{.5cm} \text{for } \xi \in K \text{ a.e.}$$
 \end{enumerate}	
\end{thm}

We deduce the following result with the help of Theorem \ref{sc}.

\begin{thm}
Let $S$ be the scaling set correspond to scaling function $\phi$ of an MRA in $L^{2}(K)$. Then $\lim\limits_{j \to \infty} \mathfrak{p}^{-j}S=K$.
\end{thm}

\begin{proof}
	From (2) of Theorem \ref{sc}, we have, for $\xi \in K$ a.e.
	\begin{eqnarray*}
		&&\lim \limits_{j \to \infty} \lvert \hat{\phi}(\mathfrak{p}^{j}\xi) \rvert =1 \\
		&\Rightarrow & \lim \limits_{j \to \infty} \lvert \mathds{1}_{S}(\mathfrak{p}^{j}\xi) \rvert =1 \\
		&\Rightarrow & \lim \limits_{j \to \infty}  \mathds{1}_{\mathfrak{p}^{-j}S}(\xi) =1 \\
		&\Rightarrow & \xi \in \mathfrak{p}^{-j}S \text{ as } j \to \infty
	\end{eqnarray*}
So $K = \mathfrak{p}^{-j}S$ a.e. when $j$ is large enough. \\
It is also to be noted from (1) of Theorem \ref{sc} that $\{S+u(t) : t \in \mathbb{N}_{0}\}$ is a partition of $\mathfrak{D}$.
\end{proof}

Shukla et al. \cite{SMM19} characterize of Parseval scaling  function in term of Parseval scaling set in following theorem.

\begin{thm}
	A function $\phi$ such that $\hat{\phi} = \mathds{1}_S$, for some measurable $S \subset K$, is a Parseval scaling function of a PMRA if and only if 
	\begin{enumerate}[label=(\alph*)]
		\item $\{ S+u(t) : t \in \mathbb{N}_{0} \}$ is a measurable partition of a subset of $K$,
		\item $\bigcup \limits_{j \in \mathbb{N}} \mathfrak{p}^{-j}S = K$,
		\item $S \subset \mathfrak{p}^{-j}S$.
	\end{enumerate}
\end{thm}

\begin{defn}(Dimension function).
	Suppose $\Psi = \{\psi^m : m = 1, 2, \ldots, L\} \subset L^{2}(K)$ is a multiwavelet for $L^{2}(K)$. The dimension function of $\Psi$ is defined as $$D_{\Psi}(\xi) = \sum \limits_{m=1}^{L} \sum \limits_{j=1}^{\infty} \sum \limits_{t \in \mathbb{N}_{0}} |\widehat{\psi^{m}}(\mathfrak{p}^{-j}(\xi+u(t)))|^2, ~~~\text{for } \xi \in K \text{ a.e.}$$
\end{defn}

It is to be observe that if $\psi^1, \ldots, \psi^L \in L^{2}(K)$, then $$\int \limits_{\mathfrak{D}} \sum \limits_{j=1}^{\infty} \sum \limits_{t \in \mathbb{N}_{0}} |\widehat{\psi^{m}}(\mathfrak{p}^{-j}(\xi+u(t)))|^2 d\xi = \sum \limits_{j=1}^{\infty} q^{-j} \int \limits_{K}|\widehat{\psi^{m}}(\xi)|^2 d\xi .$$

Gripenberg \cite{G95} and Wang \cite{W95} independently proved that a wavelet is an MRA wavelet if and only if its dimension function is 1 a.e. in real setting.

\begin{thm}
A multiwavelet set $W= W_1 \cup \ldots \cup W_{q-1} \subset K$ induces an MRA if and only if $\bigcup \limits_{m=1}^{q-1} \bigcup \limits_{j=1}^{\infty} \bigcup \limits_{t \in \mathbb{N}_{0}} \{\mathfrak{p}^{j}W_{m} +u(t)\} = K$ a.e.
\end{thm}

\begin{proof}
As $W$ is a multiwavelet set, there exists a multiwavelet $\Psi = \{\psi^m : m = 1, \ldots, q-1\}$ such that $\widehat{\psi^m} = \mathds{1}_{W_m}$.\\
We know from \cite{BJ15} that a multiwavelet is an MRA multiwavelet  if and only if $D_{\Psi} \equiv 1$ on $K$ a.e. Therefore for $\xi \in K$ a.e
\begin{eqnarray*}
&& \sum \limits_{m=1}^{q-1} \sum \limits_{j=1}^{\infty} \sum \limits_{t \in \mathbb{N}_{0}} \mathds{1}_{W_m}(\mathfrak{p}^{-j}(\xi+u(t))) =1 \\ 
&\iff& \sum \limits_{m=1}^{q-1} \sum \limits_{j=1}^{\infty} \sum \limits_{t \in \mathbb{N}_{0}} \mathds{1}_{\mathfrak{p}^{j}W_{m} +u(t)}(\xi) =1 \\ 
&\iff& \bigcup \limits_{m=1}^{q-1} \bigcup \limits_{j=1}^{\infty} \bigcup \limits_{t \in \mathbb{N}_{0}} \{\mathfrak{p}^{j}W_{m} +u(t)\} = K.
\end{eqnarray*}
This completes the proof.
\end{proof}

\section{Super-wavelet}
In this section we discuss different properties of super-wavelet and associated set.
 
 \begin{defn}(Super-wavelet).
 	Let $\Theta=\{\eta^{(1)}, \ldots, \eta^{(m)}\} \subset L^{2}(K)$ such that $\mathcal{W}(\Theta)$ is a Parseval multiframelet for $L^{2}(K)$. We call $\Theta$ a super-wavelet of length $m$ if $$\mathcal{B}(\Theta):=\{\bigoplus\limits_{i=1}^{m} q^{\frac{j}{2}}\eta^{(i)}(\mathfrak{p}^{-j}\cdot -u(t)) : j \in \mathbb{Z}, t \in \mathbb{N}_0\}$$ is an orthonormal basis for $\bigoplus \limits_{m}L^{2}(K)$. Each $\eta^{(i)}$ is called a component of the super-wavelet. Also, $\Theta$ is called a Parseval frame super-wavelet if $\mathcal{B}(\Theta)$ is a Parseval frame for $\bigoplus \limits_{m}L^{2}(K)$. 
 \end{defn}

\begin{defn}
	Two Parseval frame super-wavelets $(\eta^1, \ldots, \eta^m)$ and $(\zeta^1, \ldots, \zeta^n)$ are said to be equivalent if there exists a unitary operator $$U : \bigoplus \limits_{j=1}^m L^{2}(K) \to \bigoplus \limits_{j=1}^n L^{2}(K)$$ such that for all $k \in \mathbb{Z}$ and $l \in \mathbb{N}_0$, $$U(\eta_{k,l}^1 \oplus \cdots \oplus \eta_{k,l}^m) = \zeta_{k,l}^1 \oplus \cdots \oplus \zeta_{k,l}^n.$$
\end{defn}

\begin{defn}
	A Parseval framelet $\Upsilon$ is said to be a $n$-decomposable, where $n>1$, if $\Upsilon$ is equivalent to a Parseval frame super-wavelet of length $n$.
\end{defn}

\begin{prop}\cite{SM17}
	Let $(\psi_1, \ldots, \psi_m)$ and $(\varphi_1, \ldots, \varphi_n)$ be Parseval frame super-wavelet in $L^{2}(K)$. Then they are equivalent if and only if for $n \in \mathbb{N}_0$ and $\xi \in K$ a.e. \begin{equation} \label{eqiv1}
		\sum_{j=1}^{m} \sum_{t \in \mathbb{N}_0} \widehat{\psi_j} (\mathfrak{p}^{-n}(\xi +u(t))) \overline{\widehat{\psi_j} (\xi +u(t))} = \sum_{j=1}^{n} \sum_{t \in \mathbb{N}_0} \widehat{\varphi_j} (\mathfrak{p}^{-n}(\xi +u(t))) \overline{\widehat{\varphi_j} (\xi +u(t))}.
	\end{equation}
\end{prop}

\noindent Now considering $\widehat{\psi_j} = \mathds{1}_{W_j}$ and $\widehat{\varphi_j} = \mathds{1}_{V_j}$ for some measurable $W_j , V_j \subset K$, the equation \eqref{eqiv1} changes to 
\begin{eqnarray*}
	&&\sum_{j=1}^{m} \sum_{t \in \mathbb{N}_0} \mathds{1}_{[\mathfrak{p}^{n}W_j +u(t)]\cap [W_j +u(t)]}(\xi) = \sum_{j=1}^{n} \sum_{t \in \mathbb{N}_0} \mathds{1}_{[\mathfrak{p}^{n}V_j +u(t)]\cap [V_j +u(t)]} (\xi)\\
	&\Rightarrow& \bigcup \limits_{j=1}^m \bigcup \limits_{t \in \mathbb{N}_0} [\mathfrak{p}^{n}W_j +u(t)]\cap [W_j +u(t)] = \bigcup \limits_{j=1}^n \bigcup \limits_{t \in \mathbb{N}_0} [\mathfrak{p}^{n}V_j +u(t)]\cap [V_j +u(t)]
\end{eqnarray*}

Shukla and Maury \cite{SM17} gave following necessary condition for a Parseval framelet to be decomposable.

\begin{prop}
	If $\psi$ is a $m$-decomposable Parseval framelet then \begin{equation} \label{eq1}
	\int \limits_{\mathfrak{D}} \frac{\sum \limits_{t \in \mathbb{N}_{0}} |\hat{\psi}(\xi +u(t))|^2}{|\xi|^2} d\xi \geq m \frac{q-1}{q}.
	\end{equation}
\end{prop}

In view of $\hat{\psi} = \mathds{1}_{W}$, equation \eqref{eq1} changes to $$\int \limits_{\mathfrak{D}} \frac{\sum \limits_{t \in \mathbb{N}_{0}} \mathds{1}_{W+u(t)}(\xi)}{|\xi|^2} d\xi \geq m \frac{q-1}{q}.$$

\section*{Acknowledgment}
The author is highly indebted to the fiscal support of Ministry of Human Resource Development (M.H.R.D.), Government of India.



\end{document}